\documentclass[11pt]{amsart}
\usepackage{amssymb}
\usepackage{amsmath}
\usepackage[active]{srcltx}
\usepackage{t1enc}
\usepackage[latin2]{inputenc}
\usepackage{verbatim}
\usepackage{amsmath,amsfonts,amssymb,amsthm}
\usepackage[mathcal]{eucal}
\usepackage{enumerate}
\usepackage[centertags]{amsmath}
\usepackage{graphics}

\newtheorem{theorem}{Theorem}

\newtheorem{lemma}{Lemma}

\begin{document}
\author{G. Tutberidze}
\title[partial sums]{On the strong convergence of partial sums with respect to bounded Vilenkin systems}
\address{G.Tutberidze, The University of Georgia, School of Science and Technology, 77a Merab Kostava St, Tbilisi 0128, Georgia and Department of Engineering Sciences and Mathematics, Lule\aa\ University of Technology, SE-971 87 Lule\aa, Sweden and UiT The Arctic University of Norway, P.O. Box 385, N-8505, Narvik, Norway.}
\email{giorgi.tutberidze1991@gmail.com}
\thanks{The research was supported by a Swedish Institute scholarship and by
Shota Rustaveli National Science Foundation grant YS15-2.1.1-47.}
\date{}
\maketitle

\begin{abstract}
In this paper we investigate some strong convergence theorems for partial
sums with respect to Vilenkin system.
\end{abstract}

\date{}

\textbf{2010 Mathematics Subject Classification.} 42C10.

\noindent \textbf{Key words and phrases:} Vilenkin system, partial sums, Fejé%
r means, martingale Hardy space, strong convergence.

\section{Introduction}

It is well-known (for details see e.g. \cite{gol} and \cite{sws}) that
Vilenkin system does not form basis in the space $L_{1}\left( G_{m}\right).$
Moreover, there is a function in the Hardy space $H_{1}\left( G_{m}\right),$
such that the partial sums of $f$ \ are not bounded in $L_{1}$-norm.
However, subsequence $S_{M_{n}}$ of partial sums are bounded from the
martingale Hardy space $H_{1}\left( G_{m}\right) $ to the Lebesgue space $L_{1}\left( G_{m}\right):$
\begin{equation}  \label{1ccss}
\left\Vert S_{M_k}f\right\Vert _{H_1}\leq c\left\Vert f\right\Vert
_{H_{1}} \text{ \ \ \ } (k\in \mathbb{N}).
\end{equation}

Moreover, we have the follwing norm equivalence:
\begin{equation} \label{equi}
\left\Vert f\right\Vert_{H_{1}}\equiv \left\Vert \sup_n\left\vert S_{M_n}f\right\vert\right\Vert_{1}.
\end{equation}

Moreover, Gát \cite{gat1} proved the following strong convergence result for
all $f\in H_{1}:$%
\begin{equation*}
\underset{n\rightarrow \infty }{\lim }\frac{1}{\log n}\overset{n}{\underset{%
k=1}{\sum }}\frac{\left\Vert S_{k}f-f\right\Vert _{1}}{k}=0,
\end{equation*}%
where $S_{k}f$ denotes the $k$-th partial sum of the Vilenkin-Fourier series
of $f.$

It follows that there exists an absolute constant $c,$ such that%
\begin{equation} \label{si}
\frac{1}{\log n}\overset{n}{\underset{k=1}{\sum }}\frac{\left\Vert
S_{k}f\right\Vert _{1}}{k}\leq c\left\Vert f\right\Vert _{H_{1}} \text{ \ }%
\left( n=2,3...\right)
\end{equation}%
and 
\begin{equation*}
\underset{n\rightarrow \infty }{\lim }\frac{1}{\log n}\overset{n}{\underset{%
k=1}{\sum }}\frac{\left\Vert S_{k}f\right\Vert _{1}}{k}=\left\Vert
f\right\Vert _{H_{1}},
\end{equation*}%
for all $f\in H_{1}.$

Analogical result for the trigonometric system was proved by Smith \cite{sm}, for the Walsh-Paley system by Simon \cite{Si3}.

If partial sums of Vilenkin-Fourier series was bounded from $H_{1}$ to $L_{1} $ we also would have:
\begin{equation} \label{tut}
\underset{n\in \mathbf{\mathbb{N}}_{+}}{\sup }\frac{1}{n}\underset{m=1}{%
\overset{n}{\sum }}\left\Vert S_{m}f\right\Vert _{1}\leq c\left\Vert f\right\Vert_{H_1}.
\end{equation}
but as it was present above that boundednes of partial sums does not hold from  $H_{1}$ to $L_{1} $, However, we have inequality (\ref{si}).

On the other hand, in one-dimensional, Fujji \cite{Fu} and Simon \cite{Si2}
proved that maximal operator Fejér means is bounded from $H_{1}$ to $L_{1}$. It follows that
\begin{equation} \label{fusi}
\underset{n\in \mathbf{\mathbb{N}}_{+}}{\sup }\left\Vert \frac{1}{n}\underset{m=1}{%
\overset{n}{\sum }} S_{m}f\right\Vert _{1}<c\left\Vert f \right\Vert_{H_1}.
\end{equation}

So, natural question has arised that if inequality (\ref{tut}) holds true, which would be generalization of inequality (\ref{fusi}) or we have negative answer on this problem.

In this paper we prove that there exists a function $f\in H_{1} $ such that 
\begin{equation*}
\underset{n\in \mathbf{\mathbb{N}}_{+}}{\sup }\frac{1}{n}\underset{m=1}{%
\overset{n}{\sum }}\left\Vert S_{m}f\right\Vert _{1}=\infty.
\end{equation*}

This paper is organized as follows: in order not to disturb our discussions
later on some definitions and notations are presented in Section 2.  For the proofs of the main results we
need some auxiliary Lemmas. These results are presented in Section 3. The formulation and detailed proof of main results can be found in Section 4.

\section{Definitions and Notations}

Let $\mathbb{N}_{+}$ denote the set of the positive integers, $\mathbb{N}:=%
\mathbb{N}_{+}\cup \{0\}.$

Let $m:=(m_{0,}m_{1},\dots)$ denote a sequence of the positive integers not
less than 2.

Denote by 
\begin{equation*}
Z_{m_{k}}:=\{0,1,\dots,m_{k}-1\}
\end{equation*}
the additive group of integers modulo $m_{k}.$

Define the group $G_{m}$ as the complete direct product of the group $%
Z_{m_{j}}$ with the product of the discrete topologies of $Z_{m_{j}}$ $^{,}$%
s.

The direct product $\mu $ of the measures 
\begin{equation*}
\mu _{k}\left( \{j\}\right):=1/m_{k}\text{ \qquad }(j\in Z_{m_{k}})
\end{equation*}
is the Haar measure on $G_{m_{\text{ }}}$with $\mu \left( G_{m}\right) =1.$

If $\sup_{n\in \mathbb{N}}m_{n}<\infty $, then we call $G_{m}$ a bounded
Vilenkin group. If the generating sequence $m$ is not bounded then $G_{m}$
is said to be an unbounded Vilenkin group. \textbf{In this paper we discuss
bounded Vilenkin groups only.}

The elements of $G_{m}$ are represented by sequences 
\begin{equation*}
x:=(x_{0},x_{1},\dots,x_{k},\dots)\qquad \left( \text{ }x_{k}\in
Z_{m_{k}}\right).
\end{equation*}

It is easy to give a base for the neighbourhood of $G_{m}$ 
\begin{equation*}
I_{0}\left( x\right):=G_{m},
\end{equation*}%
\begin{equation*}
I_{n}(x):=\{y\in G_{m}\mid y_{0}=x_{0},\dots,y_{n-1}=x_{n-1}\}\text{ }(x\in
G_{m},\text{ }n\in \mathbb{N})
\end{equation*}%
Denote $I_{n}:=I_{n}\left( 0\right) $ for $n\in \mathbb{N}$ and $\overline{%
I_{n}}:=G_{m}$ $\backslash $ $I_{n}$ $.$

Let

\begin{equation*}
e_{n}:=\left( 0,\dots,0,x_{n}=1,0,\dots\right) \in G_{m}\qquad \left( n\in 
\mathbb{N}\right).
\end{equation*}

If we define the so-called generalized number system based on $m$ in the
following way: 
\begin{equation*}
M_{0}:=1,\text{ \qquad }M_{k+1}:=m_{k}M_{k\text{ }}\ \qquad (k\in \mathbb{N})
\end{equation*}%
then every $n\in \mathbb{N}$ can be uniquely expressed as $%
n=\sum_{k=0}^{\infty }n_{j}M_{j}$ where $n_{j}\in Z_{m_{j}}$ $~(j\in \mathbb{%
N})$ and only a finite number of $n_{j}`$s differ from zero. Let $\left\vert
n\right\vert :=\max $ $\{j\in \mathbb{N};$ $n_{j}\neq 0\}.$

For the natural number $n=\sum_{j=1}^{\infty }n_{j}M_{j},$ we define%
\begin{equation*}
\delta _{j}=signn_{j}=sign\left( \ominus n_{j}\right) ,\text{ \ \ \ \ }%
\delta _{j}^{\ast }=\left\vert \ominus n_{j}-1\right\vert \delta _{j},
\end{equation*}%
where $\ominus $ is the inverse operation for $a_{k}\oplus b_{k}=\left(
a_{k}+b_{k}\right) $mod$m_{k}.$

We define functions $v$ and $v^{\ast }$ by 
\begin{equation*}
v\left( n\right) =\sum_{j=0}^{\infty }\left\vert \delta _{j+1}-\delta
_{j}\right\vert +\delta _{0},\text{ \ }v^{\ast }\left( n\right)
=\sum_{j=0}^{\infty }\delta _{j}^{\ast },
\end{equation*}

Next, we introduce on $G_{m}$ an orthonormal system which is called the
Vilenkin system.

At first define the complex valued function $r_{k}\left( x\right)
:G_{m}\rightarrow \mathbb{C},$ the generalized Rademacher functions as 
\begin{equation*}
r_{k}\left( x\right):=\exp \left( 2\pi\imath x_{k}/m_{k}\right) \text{
\qquad }\left( \imath^{2}=-1,\text{ }x\in G_{m},\text{ }k\in \mathbb{N}%
\right).
\end{equation*}

Now define the Vilenkin system $\psi:=(\psi _{n}:n\in \mathbb{N})$ on $G_{m} 
$ as: 
\begin{equation*}
\psi _{n}\left( x\right):=\prod_{k=0}^{\infty }r_{k}^{n_{k}}\left( x\right) 
\text{ \qquad }\left( n\in \mathbb{N}\right).
\end{equation*}

Specially, we call this system the Walsh-Paley one if $m\equiv 2.$

The norm (or quasi norm) of the space $L_{p}(G_{m})$ is defined by \qquad
\qquad \thinspace\ 
\begin{equation*}
\left\Vert f\right\Vert _{p}:=\left( \int_{G_{m}}\left\vert f(x)\right\vert
^{p}d\mu (x)\right) ^{1/p}\qquad \left( 0<p<\infty \right) .
\end{equation*}

The Vilenkin system is orthonormal and complete in $L_{2}\left( G_{m}\right)
\,$ (for details see e.g. \cite{AVD,Vi}).

If $\ f\in L_{1}\left( G_{m}\right) $ we can establish Fourier coefficients,
partial sums of the Fourier series, Fejér means, Dirichlet kernels
with respect to the Vilenkin system in the usual manner: 
\begin{eqnarray*}
\widehat{f}(k) &:&=\int_{G_{m}}f\overline{\psi }_{k}d\mu \text{\thinspace
\qquad\ \ \ \ }\left( \text{ }k\in \mathbb{N}\text{ }\right) \\
S_{n}f &:&=\sum_{k=0}^{n-1}\widehat{f}\left( k\right) \psi _{k}\ \text{%
\qquad\ \ }\left( \text{ }n\in \mathbb{N}_{+},\text{ }S_{0}f:=0\right) \\
\sigma _{n}f &:&=\frac{1}{n}\sum_{k=0}^{n-1}S_{k}f\text{ \qquad\ \ \ \ \ }%
\left( \text{ }n\in \mathbb{N}_{+}\text{ }\right) \\
D_{n} &:&=\sum_{k=0}^{n-1}\psi _{k\text{ }}\text{ \qquad\ \ \qquad }\left( 
\text{ }n\in \mathbb{N}_{+}\text{ }\right).
\end{eqnarray*}

Recall that 
\begin{equation}  \label{3}
\quad \hspace*{0in}D_{M_{n}}\left( x\right) =\left\{ 
\begin{array}{l}
\text{ }M_{n}\text{ \ \ \ }x\in I_{n} \\ 
\text{ }0\text{ \qquad }x\notin I_{n}%
\end{array}
\right.
\end{equation}
and 
\begin{equation}  \label{9dn}
D_{s_{n}M_{n}}=D_{M_{n}}\sum_{k=0}^{s_{n}-1}\psi
_{kM_{n}}=D_{M_{n}}\sum_{k=0}^{s_{n}-1}r_{n}^{k} \text{ \qquad } 1\leq
s_{n}\leq m_{n}-1.
\end{equation}

The $n$-th Lebesgue constant is defined in the following way 
\begin{equation*}
L_{n}=\left\Vert D_{n}\right\Vert _{1}.
\end{equation*}

If $f\in L_{1}(G_{m}),$ the maximal functions are also be given by 
\begin{equation*}
f^{*}\left( x\right) =\sup_{n\in \mathbb{N}}\frac{1}{\left| I_{n}\left(
x\right) \right| }\left| \int_{I_{n}\left( x\right) }f\left( u\right) \mu
\left( u\right) \right|
\end{equation*}

Hardy martingale space $H_{1}\left( G_{m}\right)$
consist of all martingales for which (for details see e.g. \cite{We1,We3})
\begin{equation*}
\left\| f\right\| _{H_{1}}:=\left\| f^{*}\right\| _{1}<\infty.
\end{equation*}

\section{Auxiliary results}

\begin{lemma}
\label{lemma2} \cite{luko} Let $n\in \mathbb{N}$. Then 
\begin{equation*}
\frac{1}{4\lambda}v\left( n\right) +\frac{1}{\lambda}v^{\ast }\left(
n\right) +\frac{1}{2\lambda}\leq L_{n}\leq \frac{3}{2}v\left( n\right)
+4v^{\ast }\left( n\right) -1,
\end{equation*}
where $\lambda=\sup_{n\in \mathbb{N}}m_{n}. $
\end{lemma}

\begin{lemma}
\label{lemma1} \cite{smt} Let $n\in \mathbb{N}$. Then there exists an
ansolute constant $c, $ such that 
\begin{equation*}
\frac{1}{nM_{n}}\underset{k=1}{\overset{M_{n}-1}{\sum }}v\left( k\right)
\geq c>0.
\end{equation*}
\end{lemma}

\section{Main Result}

\begin{theorem}
\label{theorem2}There exists a martingale $f\in H_{1},$ such that 
\begin{equation*}
\sup_{n\in \mathbb{N}}\frac{1}{n}\overset{n}{\underset{k=1}{\sum }}%
\left\Vert S_{k}f\right\Vert _{1}=\infty .
\end{equation*}
\end{theorem}

\section{Proof of the Theorem}

\begin{proof}[Proof of Theorem \protect\ref{theorem2}]
Let  $\left\{ \alpha _{k}:k\in \mathbb{N}%
\right\} $ be an increasing sequence of the positive integers such that 
\begin{equation} \label{2a}
\sum_{k=0}^{\infty }\frac{1}{\alpha _{k}^{1/2}}<c<\infty .  
\end{equation}

Let \qquad 
\begin{equation*}
f=\sum_{k=1}^{\infty}\frac{a_{k}}{\alpha _{k}^{1/2}},
\end{equation*}%
where 
\begin{equation*} a_{k}=D_{M_{\alpha_{k}+1}}-D_{M_{_{\alpha _{k}}}}.
\end{equation*}

It is evident that
\begin{equation*}
S_{M_n}f=\sum_{\left\{ k;\text{ }\alpha _{k}<n\right\} }\frac{a_{k}}{\alpha _{k}^{1/2}},
\end{equation*}
and
\begin{equation*}
\left\vert S_{M_n}f\right\vert\leq\sum_{\left\{ k;\text{ }\alpha _{k}<n\right\}}\frac{\left \vert a_{k}\right\vert}{\alpha _{k}^{1/2}}\leq\sum_{k=1}^{\infty}\frac{\left \vert a_{k}\right\vert}{\alpha _{k}^{1/2}},
\end{equation*}

It follows that

\begin{equation*}
\sup_{n\in \mathbb{N}}\left\vert S_{M_n}f\right\vert\leq \sum_{k=1}^{\infty}\frac{\left \vert a_{k}\right\vert}{\alpha _{k}^{1/2}}.
\end{equation*}

Since (see equality (\ref{3}))
\begin{equation*}
\left\Vert a_{k}\right\Vert\leq 2, \ \ \text{fol all} \ \ k\in \mathbb{N}.
\end{equation*}
by combining (\ref{equi}) and (\ref{2a}) we get that
\begin{eqnarray*}
&&\left\Vert f\right\Vert_{H_{1}}\leq c \left\Vert \sup_{k\in \mathbb{N}}\left\vert S_{M_k}f\right\vert\right\Vert_{1} \\
&\leq & c \left\Vert \sum_{k=1}^{\infty}\frac{\left \vert a_{k}\right\vert}{\alpha _{k}^{1/2}} \right\Vert
\leq c \sum_{k=1}^{\infty}\frac{\left\Vert a_{k} \right\Vert}{\alpha _{k}^{1/2}}\leq 2c \sum_{k=1}^{\infty}\frac{1}{\alpha _{k}^{1/2}}\leq c<\infty.
\end{eqnarray*}

Moreover,
\begin{equation}  \label{6}
\widehat{f}(j)=\left\{ 
\begin{array}{l}
\frac{1}{\alpha _{k}^{1/2}},\,\,\text{\ \ \ \
\thinspace \thinspace }j\in \left\{M_{\alpha _{k}},...,M_{\alpha
_{k}+1}-1\right\}, \text{ }k\in \mathbb{N} \\ 
0\text{ },\text{ \thinspace \qquad \thinspace \thinspace \thinspace
\thinspace \thinspace }j\notin \bigcup\limits_{k=1}^{\infty }\left\{
M_{\alpha _{k}},...,M_{\alpha _{k}+1}-1\right\} .\text{ }%
\end{array}
\right.
\end{equation}

Let $M_{\alpha _{k}}\leq j<M_{\alpha _{k+1}}.$ Since 
\begin{equation*}
D_{j+M_{\alpha _{k}}}=D_{M_{\alpha _{k}}}+\psi _{_{M_{\alpha _{k}}}}D_{j},
\text{ \qquad when \thinspace \thinspace }j<M_{\alpha_{k}},
\end{equation*}%
if we apply (\ref{6}) we obtain that 
\begin{eqnarray}  \label{8}
S_{j}f &=&S_{M_{\alpha _{k}}}f+\sum_{v=M_{\alpha _{k}}}^{j-1}\widehat{f}%
(v)\psi _{v}  \\
&=&S_{M_{\alpha _{k}}}f+\sum_{v=M_{\alpha _{k}}}^{j-1}\widehat{f}(v)\psi _{v}
\notag \\
&=&S_{M_{\alpha _{k}}}f+\frac{M_{\alpha _{k}}^{1/p-1}}{\alpha _{k}^{1/2}}%
\sum_{v=M_{\alpha _{k}}}^{j-1}\psi _{v}   \\
&=&S_{M_{\alpha _{k}}}f+\frac{M_{\alpha _{k}}^{1/p-1}}{\alpha _{k}^{1/2}}%
\left( D_{j}-D_{M_{\alpha _{k}}}\right)   \\
&=&S_{M_{\alpha _{k}}}f+\frac{M_{\alpha _{k}}^{1/p-1}}{\alpha _{k}^{1/2}}%
\psi _{M_{\alpha _{k}}}D_{j-M_{\alpha _{k}}} \\ 
&=&III_{1}+III_{2}.
\end{eqnarray}

In view of (\ref{1ccss}) we can write that 
\begin{eqnarray} \label{8bbb2}
\left\Vert III_{1}\right\Vert_{1}
\leq\left\Vert S_{M_{\alpha_{k}}}f\right\Vert _{1}  \leq c\left\Vert f\right\Vert_{H_{1}}. 
\end{eqnarray}

By combining Lemma \ref{lemma2} and (\ref{8bbb2}) we get that%
\begin{eqnarray*}
\left\Vert S_{n}f\right\Vert _{1} 
\geq \left\Vert III_{2}\right\Vert _{1}-\left\Vert III_{1}\right\Vert _{1}
\geq\frac{cv\left( {n-M_{\alpha _{k}}}\right)}{\alpha _{k}^{1/2}}-c\left\Vert f\right\Vert _{H_{1}}.
\end{eqnarray*}

Hence, according to Lemma \ref{lemma1} we can conclude that 
\begin{eqnarray*}
&&\underset{n\in \mathbf{\mathbb{N}}_{+}}{\sup }\frac{1}{n}\underset{k=1}{%
\overset{n}{\sum }}\left\Vert S_{k}f\right\Vert _{1} \\
&\geq &\frac{1}{M_{\alpha _{k}+1}}\underset{\left\{ M_{\alpha _{k}}\leq
l\leq 2M_{\alpha _{k}}\right\} }{\sum }\left\Vert S_{l}f\right\Vert _{1} \\
&\geq &\frac{1}{M_{\alpha _{k}+1}}\underset{\left\{ M_{\alpha _{k}}\leq
l\leq 2M_{\alpha _{k}}\right\} }{\sum }\left( \frac{v\left( l-M_{\alpha
_{k}}\right) }{\alpha _{k}^{1/2}}-c\left\Vert f\right\Vert _{H_{1}}\right) \\
&\geq &\frac{c}{\alpha _{k}^{1/2}M_{\alpha _{k}}}\underset{l=1}{\overset{%
M_{\alpha _{k}}-1}{\sum }}v\left( l\right) -c\left\Vert f\right\Vert
_{H_{1/2}}^{1/2} \\
&\geq &c\alpha _{k}^{1/2}\rightarrow \infty ,\text{ as \ }k\rightarrow
\infty .
\end{eqnarray*}

The proof is complete.
\end{proof}

\qquad

\end{document}